\documentclass[12pt]{amsart}
\usepackage{amsfonts}
\usepackage{amsthm}
\usepackage{mathabx}
\usepackage{wasysym}
\usepackage{mathrsfs}
\usepackage{amssymb}
\usepackage{textcomp}
\usepackage{eucal}
\newtheorem{Pa}{Paper}[section]
\newtheorem{theorem}[Pa]{{\bf Theorem}}
\newtheorem{lem}[Pa]{{\bf Lemma}}

\newtheorem{Rk}[Pa]{{\bf Remark}}
\newtheorem{prop}[Pa]{{\bf Proposition}}
\newtheorem{Ex}[Pa]{{\bf Example}}
\def\R{\mathbb R}

\def\(s){\mathscr S(\R^2)}

\usepackage{appendix}
\usepackage{amsmath}
\usepackage{color}
\usepackage{graphicx}
\usepackage{float}
\usepackage{amssymb}
\usepackage[update,prepend]{epstopdf}
\usepackage{graphicx}
\usepackage{caption}
\usepackage{subcaption}
\usepackage{scalerel}

\title[]
{Wiener-Chaos approach to optimal prediction}

\author{Daniel Alpay}
\address{(DA) Department of Mathematics
\newline
Ben Gurion University of the Negev \newline P.O.B. 653,
\newline
Be'er Sheva 84105, \newline ISRAEL}
\email{dany@math.bgu.ac.il}
\author{Alon Kipnis}
\address{(AK) Department of Electrical Engineering
\newline
Stanford University \newline
Stanford, California \newline }
\email{kipnisal@stanford.edu}
\thanks{D. Alpay thanks the
Earl Katz family for endowing the chair which supported his
research}

\begin{document}

%\tableofcontents
\begin{abstract}
The chaos expansion of a general non-linear function of a Gaussian stationary increment process conditioned on its past realizations is derived. This work combines Wiener chaos expansion approach to study the dynamics of a stochastic system with the classical problem of the prediction of a Gaussian process based on a realization of its past. This is done by constructing a special basis for the Fock space of the Gaussian space generated by the process, such that each basis element is either measurable or independent with respect to the given samples. This property of the basis allows us to derive the chaos expansion of a random variable conditioned on part of the sample path. We provide a general method for the construction of such basis when the underlying process is Gaussian with stationary increment. We evaluate the basis elements in the case of the fractional Brownian motion, which leads to a prediction formula for this process.
\end{abstract}

\maketitle


\section{Introduction}
\setcounter{equation}{0}
The Wiener chaos expansion has gained popularity in the recent years as a tool to study the dynamic of stochastic systems \cite{new_sde, Kah00, Hou06wienerchaos, briand2014simulation, lototsky2006stochastic}. In this approach, the randomness is due to a fundamental set of Gaussian random variables and the probability space is decomposed into a direct sum of spaces spanned by polynomials in these Gaussian elements. \par
Our starting point is a second order Gaussian stationary-increment process $X(\cdot)=\left\{X(t),\,\, t\in\mathbb R\right\}$, over the probability
space $\mathbf L_2(\Omega,\mathscr F,\mathbb P)$, with spectral measure $\Delta$. Let $\mathcal G$ denote the Gaussian Hilbert space generated by this process, and $\mathscr F_A$ the sigma field induced by $\left\{X(t),\,\, t\in A\right\}$, where $A$ is a Borel set. Given $Y\in\mathbf L_2(\Omega,\mathscr F_{\mathbb R}, \mathbb P)$, we present a new way to obtain the conditional expectation $\mathbb E\left[Y|\mathscr F_A\right]$
which is based on the Wiener chaos written in terms of a special basis for $\mathcal G$ and the Hermite polynomials. Writing
\begin{equation}
\label{eq:Y_chaos}
Y(t)=\sum_{\alpha\in \mathcal J}y_\alpha(t) H_\alpha
\end{equation}
we have
\begin{equation} \label{eq:general_idea}
\mathbb E\left[Y(t)|\mathscr F_A\right] = \sum_{\alpha\in \mathcal J_0}y_\alpha(t) H_\alpha,
\end{equation}
where $\mathcal J_0\subset \mathcal J$ depends only on $A$ and $\Delta$. That is, each chaos element $H_\alpha$ is either measurable with respect to $\mathscr F_A$ or independent with respect to it.\\

Wiener chaos expansion has been found useful in studying the dynamics of a stochastic system driven by an underlying Gaussian process. In particular, a significant attention was given to its use in stochastic differential equations (SDE), since in many cases it is more feasible to obtain an expression for the chaos expansion of the solution to a SDE rather than the solution itself \cite{lototsky2006stochastic}. For example, in some cases an approximate solution can be obtained by deriving a series of regular differential equations \cite{Hou06wienerchaos}. 
The representation \eqref{eq:general_idea} developed in this work provides an easy way to obtain the solution of SDE when past realization of the noise process are given: the conditioned solution is simply given by discarding those coefficients which are not belong to $\mathcal J_0$. When the underlying process is a semi-martingale, the conditional solution is usually a simple function of the most recent sample (although it may not be trivial to compute, see \cite{briand2014simulation} for an example). The significance of the representation \eqref{eq:general_idea} is primarily when the underlying randomness is due to a general stationary-increment process with a richer memory structure than a semi-martingale. The special case of the fractional Brownian motion will be given a special attention. \\

If $Y$ in \eqref{eq:Y_chaos} belongs to $\mathcal G$, the problem of computing $\mathbb E\left[Y|\mathscr F_A\right]$ reduces to the problem of orthogonal projection onto the closed linear span of the functions $\left\{X(s),\,\, s\in A\right\}$, denoted by $\mathcal G_A$. In particular,  when $Y=X(t)$ with $t>0$ and $A=(-\infty,0]$ this is the classical Wiener-Kolmogorov prediction problem. If $t>T$ and $A=[-T,T]$, this is the finite horizon prediction problem, which was solved by Krein; see \cite{Dmk}. In this work, we consider these two cases and employ similar methods to obtain an orthonormal basis for the space $\mathcal G_A$. We also note that another case of interest is the interpolation problem, when $t\in(-T,T)$ and $A=(-\infty-T]\cup[T,\infty)$. This problem was solved by Dym and McKean \cite{Dmk}. We also refer to this book for background material on these various problems. \par
Consider now the more general case where $Y$ does not belong to the closed linear span of $X(\cdot)$. In this case the prediction problem becomes a non-linear problem, and it is usually hard to evaluate the predicted value of $Y$ from the statistics of the underlying process $X(\cdot)$. For example, if $Y=f(U)$ where $U\in \mathcal G$ and $f:\mathbb R \rightarrow \mathbb R$ measurable function such that $f(U) \in \mathbf L_2\left(\Omega,\mathscr F(\mathcal G),\mathbb P\right)$, it follows from \cite{Howland1979} that
\begin{equation} \label{eq:Howland}
\mathbb E \left[ f(U)| \mathscr F_A\right]=\int_{\mathcal G} f\left( \mathbb E[U| \mathscr F_A](w)+(I-P)(w)  \right) d\mu(w),
\end{equation}
where $d\mu$ is the standard Gaussian measure on $\mathcal G$ and $P$ is the orthogonal projection from  $\mathcal G$ onto $\mathcal G_A$. If $\mathcal G_A$ is one dimensional, then \eqref{eq:Howland} reduces to
\[
 \mathbb E \left[ f(U)| \mathscr F_A\right]=f_M\left( \mathbb E[U| \mathscr F_A] \right),
\]
where $f_M(U)$ is the Mehler transform of $f$; see \cite[Ex. 4.18]{MR99f:60082}. In general, formula \eqref{eq:Howland} does not lead to easy computations because of the Gaussian integral. \\

The purpose of the approach presented in this work is to transform the non-linear prediction problem $\mathbb E \left[f(U)| \mathscr F_A \right]$ into the linear problem of finding  $\mathbb E\left[U|\mathscr F_A \right]$. As an example where such transformation is easily obtained, consider the case where $f(x)=x^n$, or simply  $Y=U^n$ for $U\in \mathcal G$ and a positive integer $n$. In this case $\mathbb E\left[Y|\mathscr F_A\right]$ is the $n_{th}$ moment of the the random variable $U$ with respect to the conditional Gaussian distribution 
\[
f_{U|\mathcal F_A}(u)=\frac{1}{\sqrt{2\pi \sigma_{MSE}^2}} \exp\left\{-\frac{(u-\mathbb E\left[U|\mathscr F_A \right])^2}{2\sigma_{MSE}^2} \right\},
\]
where $\sigma_{MSE}^2=\mathbb E\left[(U-\mathbb E\left[U|\mathscr F_A \right])^2\right]$. We have
\begin{equation} \label{eq:simple_non_linear}
\mathbb E\left[U^n|\mathscr F_A\right]= h_n^{[-\sigma_{MSE}^2]}\left(\mathbb E\left[U|\mathscr F_A \right] \right),
\end{equation}
where $h_n^{[\alpha]}(x)$ is the $n_{th}$ Hermite polynomial with parameter $\alpha$:
\begin{equation}\label{eq:hermite_poly_param}
h_n^{[\alpha]}\left(x\right)\triangleq  n! \sum_{m=0}^{\lfloor n/2 \rfloor}  \frac{\alpha^{m}x^{n-2m} \cdot \left(-1/2\right)^m}{m!(n-2m)!}.
\end{equation}
Relation \eqref{eq:simple_non_linear} can be reformulate as
\begin{equation} \label{eq:simple_non_linear2}
\mathbb E\left[ h_n(U)|\mathscr F_A \right]=h_n^{[\sigma_{MSE}^2]}\left( E\left[U|\mathscr F_A \right] \right),
\end{equation}
where we denote $h_n=h_n^{[1]}$. Note that \eqref{eq:simple_non_linear} and \eqref{eq:simple_non_linear2} effectively transformed the  non-linear prediction problem into the Wiener-Kolmogorov-Krein linear prediction problem. This approach can be generalized by decomposing an element $Y\in \mathbf L_2\left(\Omega,\mathscr F,\mathbb P\right)$ as a sum of polynomials in elements of $\mathcal G$. This is the idea behind the Wiener chaos decomposition. \\

Let $\left\{E_k,k\in  J \right\}$, where $J\subset \mathbb Z$, be an orthogonal basis for the Gaussian Hilbert space $\mathcal G$. The Wiener-chaos expansion with respect to this basis is a decomposition of the space $\mathbf L_2\left(\Omega,\mathscr F(\mathcal G),\mathbb P\right)$ into spaces of polynomials, obtained as follows \cite{MR99f:60082, Holden, HidaKuo}: Denote by $\mathcal J$ be the set of multi-indexes over $J$, i.e. the set of functions $J \rightarrow \mathbb N$ with compact support. For $\alpha=\left(...\alpha_{j_1},\alpha_{j_2},... \right)\in \mathcal J$, define
\[
H_{\alpha}(\omega)=\prod_{j \in J} h_{\alpha_j}\left( E_j(\omega) \right),
\]
where $\left\{ h_n ,n\geq 0\right\}$ are the Hermite polynomials
\begin{equation}\label{eq:hermite_poly}
h_n\left(x\right)\triangleq  n! \sum_{m=0}^{\lfloor n/2 \rfloor}  \frac{x^{n-2m} \cdot \left(-1/2\right)^m}{m!(n-2m)!}.
\end{equation}

Assume moreover that $\left\{E_j,j\in J_0\right\}$, $J_0 \subset J$ is an orthogonal basis for $\mathcal G_A$  and denote by $\mathcal J_0$ the subset of multi-indexes whose support is contained in $J_0$.
Our underlying observation is given by the following theorem.

\begin{theorem} \label{th:main}
For every $\alpha \in \mathcal J$, $H_\alpha$ is measurable with respect to $\mathscr F_A$ if and only if $\alpha \in \mathcal J_0$, i.e. the support of $\alpha$ is contained in $J_0 \subset \mathbb Z$.
\end{theorem}
Theorem \ref{th:main} can be implicitly found in \cite[Ch. 7]{MR99f:60082}, and an explicit proof will be given in Section 2. Theorem \ref{th:main} might have been useless unless we could obtain some explicit orthogonal bases for the space $\mathcal G_A$, this is the content of Sections 4 and 5 in which we review some methods to do so in two cases of interest for the time index set $A$. The setting for Sections 4 and 5 is given in Section 2. In Section 6 we discuss on application of these chaos elements, and provide explicit evaluation for the case where $X(\cdot)$ is the fractional Brownian motion.

\section{Proof of Theorem \ref{th:main} \label{sec:proof_main}}

\begin{theorem} \label{th:main2}
Let $\left\{E_k,k\in  J \right\}$ be an orthonormal basis for $\mathcal G$ such that $\left\{E_j,j\in J_0\right\}$ span $\mathcal G_A$ where $J_0 \subset J$. Then for every $\alpha \in \mathcal J$, $H_\alpha$ is measurable with respect to $\mathscr F_A$ if and only if $\alpha$ is contained in $\mathcal J_0 \subset \mathbb Z$.
\end{theorem}

\begin{proof}
Let $\Gamma(\mathcal G)$ be the symmetric Fock space of $\mathcal G$. Recall that we have \cite[p. 18]{MR99f:60082}
\[
\Gamma(\mathcal G)=\bigoplus_{n=0}^\infty \mathcal G^{:n:}=\mathbf L_2(\Omega,\mathscr F_{\mathbb R},\mathbb P),
\]
where $\mathcal G^{:n:}$ is the $n_{th}$ symmetric tensor power of $\mathcal G$. We also denote by
 \[
 :X_1,\ldots,X_m:
 \]
the Wick product of the elements $X_1,\ldots,X_m$ of $\mathcal G$.

Let $P$ denote the orthogonal projection onto $\mathcal G_A$, i.e. for an element $X\in \mathcal G$ we have
\[
PX=\mathbb E\left[ X|\mathscr F_A \right].
\]

Since $\|P\|=1$, $\Gamma P$, the second quantization of $P$, is a bounded linear operator on $\Gamma(\mathcal G)$ \cite[Theorem 4.5]{MR99f:60082}, and by \cite[Theorem 4.9]{MR99f:60082} we have that $\Gamma P: \mathbf L_2\left(\Omega,\mathscr F_{\mathbb R},\mathbb P\right) \rightarrow \mathbf L_2(\Omega,\mathscr F_A,\mathbb P)$ equals the conditional expectation
\[
Y \rightarrow \mathbb E\left[Y|\mathscr F_A \right],\quad Y\in \mathbf L_2(\Omega,\mathscr F_{\mathbb R},\mathbb P).
\]
It follows that
\begin{equation} \label{eq:thm1_proof}
\begin{split}
\mathbb E\left[ H_\alpha |\mathscr F_A\right] & = (\Gamma P) H_\alpha \\
& =  \Gamma P \prod_{j \in J}  h_{\alpha_j} \left(E_j \right)\\
& \stackrel{a} = \Gamma P : \prod_{j \in  J}  E_j^{\alpha_j}: \\
& \stackrel{b} = ~ : \prod_{j \in J}  P E_j^{\alpha_j} :\\
& = \prod_{j \in  J}  h_{\alpha_j} \left(P E_j \right)\\
& = \prod_{j \in J_0}  h_{\alpha_j} \left(P E_j \right) ~\prod_{j \in  J \setminus J_0}  h_{\alpha_j} \left(P E_j \right),
\end{split}
\end{equation}
where $(a)$ follows from \cite[Theorem 3.21]{MR99f:60082} and (b) follows from the definition of the second quantization of $P$ (see \cite[Theorem 4.5]{MR99f:60082}). Since $E_j \in \mathcal G_A^\bot$ for $j\in J\setminus J_0$, if $\alpha_j \neq 0$ for some entry of $\alpha$ then $h_{\alpha_j}(PE_j)=h_{\alpha_j}(0) = 1$. In this case $\mathbb E\left[H_\alpha|\mathscr F_A \right]=0$, which means that $H_\alpha$ is independent of $\mathscr F_A$. The other option is that $\alpha_j=0$ for all $j\in J\setminus J_0$. Since $h_{\alpha_j}(x)\equiv 1$ when $\alpha_j=0$ by its definition, \eqref{eq:thm1_proof} implies
\[
\mathbb E\left[H_\alpha|\mathscr F_A \right]=\prod_{j \in J_0}  h_{\alpha_j} \left(P E_j \right)= \prod_{j \in J_0}  h_{\alpha_j} \left(E_j \right) = H_\alpha.
\]
\end{proof}

\section{Hilbert spaces associated with a Gaussian stationary increment process}
\label{sec:3}
In this section we review standard ideas from the literature on continuous time Gaussian stochastic processes. We describe two additional Hilbert spaces isomorphic to $\mathcal G$, using the notions of the \textit{Wiener integral} and the \textit{trigonometric isomorphism}.
This sets the frameworks for sections 4 and 5 in which we obtain a basis for $\mathcal G$ that satisfy the conditions in Theorem \ref{th:main2}. \\

Assume first we are given a Gaussian stationary process $\dot{X}(\cdot)\triangleq \left\{\dot{X}(t),t\in \mathbb R\right\}$. The spectral measure $\Delta(\gamma)$ is determined by Bochner's theorem through
\begin{equation} \label{eq:stationary_cov}
\mathbb E \left[\dot{X}(t_1)\dot{X}(t_2)\right]=\int_{-\infty}^\infty e^{i\gamma(t_1-t_2)}d\Delta(\gamma).
\end{equation}
This defines the so called \textit{trigonometric isomorphism} between the Gaussian Hilbert space $\mathcal G$ generated by $\dot{X}(\cdot)$, i.e. the close linear span of $\left\{\dot{X}(t),t\in \mathbb R \right\}$ in $\mathbf L_2\left(\Omega, \mathscr F_{\mathbb R},\mathbb P\right)$, and the space $\mathbf L_2\left(d\Delta \right)$, given by
\[
\dot{X}(t)\longrightarrow e^{i\gamma t}.
\]

If $\dot{X}(\cdot)$ is path-wise integrable then $X(t)\triangleq\int_0^t \dot{X}(s)ds$ is a Gaussian stationary increment process, with covariance function
\begin{equation} \label{eq:cov}
\mathbb E\left[X(t_1)X(t_2)\right]=\int_{-\infty}^\infty  \frac{1-e^{i\gamma t_1}}{i\gamma} \frac{1-e^{-i\gamma t_2}}{-i\gamma} d\Delta(\gamma),
\end{equation}
so that $\frac{e^{i\gamma t}-1}{i\gamma }$ is the image of $X(t)$ under the trigonometric isomorphism. In the setting of distributions, we can write the following correspondences
\begin{equation} \label{eq:correspondance}
\begin{split}
\dot{X}(t)  \longleftrightarrow & \quad e^{i\gamma t}  \quad \longleftrightarrow  \delta(t-\cdot),  \\
X(t)  \longleftrightarrow   & ~~ \frac{e^{i\gamma t-1}}{i\gamma}  ~~ \longleftrightarrow  \mathbf 1_t,
\end{split}
\end{equation}
where the left relation is the trigonometric isomorphism and the right relation is the Fourier transform. In \eqref{eq:correspondance} we used $\delta(t)$ to denote the Dirac delta distribution concentrated at the origin and
\[
\mathbf 1_t(x) \triangleq \mathbf 1_{[0,t]}(x) \triangleq \begin{cases} 1, & x\in [0,t] \\
0, & x\notin [0,t],
\end{cases}.
\]
We see that for a given $t\geq 0$, $X(t)$ may be interpreted as the stochastic integral of the deterministic function $\mathbf 1_t$ \cite[p. 87]{MR99f:60082}, and can be extended to $t<0$ by setting
\[
\mathbf 1_t(x) = \begin{cases} 1 & 0< x \leq t, \\
      -1 & -t\leq x < 0, \\
   0 & {\rm otherwise} \end{cases}.
 \]

For $f\in \mathbf L_2(\mathbb R)$ we denote by $\widehat{f}$ its Fourier transform
\[
 \widehat{f}(\gamma)=\int_{-\infty}^\infty f(t)e^{i\gamma t}
\]
and by $\widecheck{f}$ its inverse. \\

In general, for $f\in \mathbf L_2(\mathbb R)$ subject to
\begin{equation} \label{eq:hilbert_space_condition}
\int_{\mathbb R} |\widehat{f}(\gamma)|^2d\Delta(\gamma) < \infty,
\end{equation}
we can define its stochastic integral with respect to the process $X(\cdot)$ as the zero mean Gaussian random variable with variance $\int_{\mathbb R} |\widehat{f}(\gamma)|^2d\Delta(\gamma)$. The set of functions in $\mathbf L_2(\mathbb R)$ which satisfy \eqref{eq:hilbert_space_condition} constitute a pre-Hilbert space, and we denote its completion by $\mathbf L_\Delta$. The map $I:\mathbf L_\Delta \longrightarrow \mathcal G$ that carries an element of $\mathbf L_\Delta$ into its stochastic integral is an Hilbert space isomorphism, and $\mathcal G$ can be regarded as the Gaussian Hilbert spaces associated with the Hilbert space $\mathbf L_\Delta$ \cite{MR99f:60082}, in the sense that for each $f_1,...,f_n \in L_\Delta$, $I(f_1),...,I(f_n)$ have a joint central normal distribution with covariance matrix $Q$,
\[
Q_{j,i}=\left(f_i,f_j\right)_{\Delta}=\int_{-\infty}^\infty \widehat{f}_i(\gamma)  \overline{\widehat{f}_j}(\gamma)d\Delta(\gamma),
\]
where $\left(\cdot,\cdot\right)_\Delta$ is the inner product in $\mathbf L_\Delta$ induced by the norm \eqref{eq:hilbert_space_condition}. Using these notations, the covariance function \eqref{eq:cov} can be written as
\[
\mathbb E\left[X(t) X(s)\right]=\left(\mathbf 1_t,\mathbf 1_s\right)_\Delta.
\]
In the case of $d\Delta=d\gamma$, $\mathbf L_\Delta$ reduces to $\mathbf L_2(\mathbb R)$ and the image of $f\in L_2(\mathbb R)$ under $I$ is called the Wiener stochastic integral of $f$ \cite[Chapter 9]{doob1953stochastic}. \\

\begin{Rk} \label{rk:pathwise_integral}
We note that for $f\in \mathbf L_\Delta$, one can define its stochastic integral with respect to $X$ in the usual way by first setting
\[
\int_{\mathbb R} f(t) dX(t)=\sum_i \alpha_i \left( X(t_{i+1})-X(t_i)\right)
\]
for a simple function $f(t)=\sum_i \alpha_i \mathbf 1_{[t_{i+1},t_i]}$, and then take the limit in $\mathbf L_\Delta$ for a general $f\in \mathbf L_\Delta$. It can be shown that we obtain
\begin{equation} \label{eq:Wiener_integral}
\int_{\mathbb R}f(t)dX(t)= I(f)
\end{equation}
in $\mathbf L_2\left(\Omega,\mathscr F_{\mathbb R},\mathbb P\right)$, that is, both definitions coincide.
\end{Rk}

In many practical cases, almost every sample path of the stationary increment process $X(\cdot)$ is nowhere differentiable. This happens for example in the case of the Brownian motion or the fractional Brownian motion. However under the condition
\begin{equation} \label{eq:measure_growth_condition}
\int_{\mathbb R} \frac{d\Delta(\gamma)}{1+\gamma^2} <\infty,
\end{equation}
it is easy to see that the indicator function $\mathbf 1_t$ still belongs to $\mathbf L_\Delta$. Since both spaces $\mathcal G$ and $\mathbf L_\Delta$, as well as the isometric map between them, are determined exclusively by the spectral measure $\Delta$, starting with $\Delta$, we may use the representation
\[
X(t)=I(\mathbf 1_t),\quad t\in \mathbb R
\]
as the definition of the process $X$. We set
\[
z_t=T(\mathbf 1_t)=\frac{e^{i\gamma t-1}}{i\gamma},\quad t\in \mathbb R.
\]
Under the condition \eqref{eq:measure_growth_condition}, each $z_t$ belongs to $\mathbf L_2(d\Delta)$. Denote by $\mathbf Z$ the close linear span of $\left\{z_t,~t\in \mathbb R\right\}$ in $\mathbf L_2(d\Delta)$ and by $\mathbf Z_A$ the close linear span of $\left\{z_t,~t \in A\right\}$ in $L_2(d\Delta)$. It is well known (see for example \cite{Dmk} that $\mathbf Z=\mathbf L_2(d\Delta)$. \\

We have obtained the following isomorphic Hilbert spaces
\[
\mathbf L_\Delta \xrightarrow{ \quad I \quad } \mathcal G \xrightarrow{ \quad T \quad } \mathbf Z.
\]
Note that in the sense of distributions,  $T\circ I: \mathbf L_\Delta \rightarrow \mathbf Z$ is the Fourier transform. \\
The importance of the above Hilbert spaces isomorphism is that it allows us to exchange the problem of obtaining an orthogonal basis for $\mathcal G$ and $\mathcal G_A$ with the problem of doing so in $\mathbf Z$ and $\mathbf Z_A$. Our benefit comes from the fact that now the theory of orthogonal projections into spaces of analytical functions is at our disposal.

\section{Prediction with respect to the entire past}
\label{sec:entire_past}
In order to be in the setting of Theorem \ref{th:main}, we first need to  find an explicit orthogonal basis for the space $\mathcal G_A$. In this section we will show how to do so in the case that $A=\left(-\infty,0\right]$ which corresponds to the Wiener-Kolmogorov prediction problem. The case where $A=\left[-T,T\right]$ for some $T>0$ which corresponds to the problem solved by Krein is treated in the next section.\\

Recall that in the case of prediction with respect to the entire past, Szeg\"{o} theorem provides us with a criterion whether the prediction is perfect or not, i.e. when
\[
\mathbb E\left[X(t)| \mathscr F_{\left(-\infty,0\right]}\right]=X(t),\quad \forall t\in \mathbb R,
\]
is in $\mathbf L_2(\Omega,\mathscr F_{\mathbb R},\mathbb P)$ or not. Or in trigonometric language: whether $z_t\in \mathbf Z_{(-\infty,0]}$ or else
\[
\mathbb E\left[\left(X(t)-\mathbb E\left[X(t)| \mathscr F_{\left(-\infty,0\right]}\right]\right)^2 \right]>0.
\]
Szeg\"{o} criterion says that if
\begin{equation} \label{eq:szego1}
\int_{-\infty}^\infty \frac{\log\Delta'(\gamma)}{\gamma^2+1}d\gamma>-\infty,
\end{equation}
then $\mathbf  Z_{(-\infty,0]} \neq \mathbf Z$, and in particular $z_T \notin \mathbf Z_{(-\infty,0]}$ for any $T>0$. 
The other option 
\begin{equation}
\int_{-\infty}^\infty \frac{\log\Delta'(\gamma)}{\gamma^2+1}d\gamma = -\infty,
\end{equation}
implies $\mathbf Z_{(-\infty,0]} =\mathbf  Z$, i.e., the future is completely determined by the past.

Karhunen \cite{karhunen1950struktur} has showed that under the conditions \eqref{eq:szego1} and $\Delta(\infty)=\int_{-\infty}^\infty d\Delta(\gamma) <\infty$, the spectral density can be decomposed as
\begin{equation} \label{eq:spectral_dec}
\Delta'(\gamma)=h(\gamma)\overline{h}(\gamma),
\end{equation}
where $h$ is an \textit{outer function} in the Hardy space $H^{2+}$ (see \cite{MR0027954}). An outer function $h\in H^{2+}$ satisfies the property that the span of $e^{i\gamma t}\overline{h}(\gamma)$, $t<0$, in $\mathbf L_2(d\gamma)$ equals $H^{2-}$, 
or equivalently, that the
span of $z_t \overline{h}$, $t<0$, in $\mathbf L_2(\mathbb R)$ equals $H^{2-}$. To see this equivalence we note that, for $k\in\mathbf L_1(\mathbb R)$
\[
\int_{\mathbb R}e^{i\gamma t}k(\gamma)d\gamma=0,\forall t>0\,\,\iff\,\,
\int_{\mathbb R}\frac{e^{i\gamma t}-1}{\gamma}k(\gamma)d\gamma=0, \forall t>0
\]
as is seen by differentiation and integration with respect to $t$.\\
%This spectral decomposition was first used by Wiener and Hopf, as a method to solve an %equation with the same name, which was later turned out to be a key equation in linear %filtering \cite[Chapter 2, Equation 11]{Kailath three decades}.  \\

Throughout this section we assume the spectral measure $\Delta$ satisfies both condition \eqref{eq:measure_growth_condition} and Szeg\"{o} criterion for optimal prediction with respect to the entire past. We further assume that $d\Delta$ is absolutely continuous with respect to the Lebesgue measure, namely $d\Delta=\Delta'(\gamma)d\gamma$. In view of the discussion in \cite[Section 4.3]{Dmk}, this assumption does not limit the generality of our approach. Since these assumption does not yet guaranty $\Delta(\infty)<\infty$, we look instead at the measure $\frac{\Delta'(\gamma) d\gamma}{1+\gamma^2}$. We have
\[
\int_{-\infty}^\infty \frac{\log\left(\frac{\Delta'(\gamma)}{1+\gamma^2} \right)}{1+\gamma^2} d\gamma < \infty,
\]
so that we may decompose $\frac{\Delta'(\gamma) d\gamma}{1+\gamma^2}$ as
\begin{equation} \label{eq:def_h}
\frac{\Delta'(\gamma) d\gamma}{1+\gamma^2} = h(\gamma)\overline{h}(\gamma)
\end{equation}
with $h$ outer, and thus obtain the decomposition:
\[
d\Delta=\Delta'(\gamma)d\gamma=|\left(\gamma-i\right)h(\gamma)|^2d\gamma.
\]
Since the process $X$ is real, $\Delta$ is always even and we can also impose the condition $h(-\gamma)=-\overline{h}(\gamma)$ \cite[Exercise 2.7.4]{Dmk}, so that the inverse Fourier transform of $(\gamma-i)h$ is a real distribution.\\

\begin{lem} \label{lem:span}
The closed linear span of the functions $\left\{(\gamma+i)z_th,\,\,t\leq 0 \right\}$ in $\mathbf L_2(d\gamma)$ equals $H^{2-}$.
\end{lem}
\begin{proof}
Denote by $\mathbf K$ the close linear span of $\left\{z_t(\gamma+i)h,\,\,t\leq 0 \right\}$ in $\mathbf L_2(d\gamma)$.
Let $\mathbf L$ be the set of functions $f\in H^{2-}$ such that $(\gamma-i)f(\gamma)$ still belongs to $H^{2-}$. It has been noted above that for $h\in H^{2-}$ outer, the closed linear span of $\left\{z_th,\,\,t \leq 0 \right\}$ in $\mathbf L_2(d\gamma)$ is all $H^{2-}$, so if a function $f\in \mathbf L$ satisfies
\[
\int_{-\infty}^\infty z_t(\gamma)(\gamma+i)h(\gamma)\overline{f}(\gamma)d\gamma=0,
\]
for all $t\leq 0$, we conclude that $(\gamma+i)f \equiv 0$, so $f=0$, and we have that $\mathbf K^\perp \bigcap \mathbf L=\left\{0\right\}$. In order to complete the proof it is enough to show that $\mathbf L$ is a dense subset of $H^{2-}$. The Schwartz space $\mathscr S$ of smooth rapidly decreasing functions is a dense subset of $\mathbf L_2(\mathbb R)$ and is invariant under differentiation and Fourier transformation. For $s\in \mathscr S$ we denote by $\widetilde{s}$ its projection into $\mathbf L_2\left( (-\infty,0]\right)%%
$, which is the set of functions in $\mathbf L_2(\mathbb R)$ supported in $\left(-\infty,0\right]$. Recall that $H^{2-}=\widehat{\mathbf L_2}\left( (-\infty,0]\right) = \left\{ \widehat{f},\,f\in \mathbf L_2\left((-\infty,0] \right) \right\}$, which implies that $\widehat{\widetilde{\mathscr S}}$ is a dense subset of $H^{2-}$. We prove that it is also contained in $\mathbf L$. Let $s\in \mathscr S$, then
\[
 \left(\gamma+i\right)\widehat{\widetilde{s}}= \left(\gamma+i\right)\int_{-\infty}^0 s(t)e^{i\gamma t}dt= i\int_{-\infty}^0 \frac{ds}{dt}(t)e^{i\gamma t}dt+i\int_{-\infty}^0 s(t)e^{i\gamma t}dt,
\]
and the last two terms are the Fourier transform of functions in  $\mathbf L_2\left( (-\infty,0]\right)$, hence belong to $H^{2-}$. This completes the proof.
\end{proof}

In what follows we construct an orthonormal basis for the space $\mathbf L_\Delta$ in terms of the function $h$ and the functions
\[
e_n(\gamma)=\frac{1}{\sqrt{\pi}} \frac{1}{1-i\gamma}\left(\frac{1+i\gamma}{1-i\gamma}\right)^n, n\in \mathbb Z,
\]
which constitute an orthonormal basis of $\mathbf L_2(\mathbb R)$. In addition, the family $\left\{e_n,\,\, n\geq 0\right\}$ spans the Hardy space $H^{2+}$ while the family $\left\{e_n ,\,\, n<0\right\}$ spans $H^{2-}$; see \cite[Section 2.5]{Dmk}.
We also note that for $n\geq 0$ and $x>0$, the inverse Fourier transform of the $\{e_n\}$ are the Laguerre functions:
\[
e_n^{\vee}(x)=\frac{1}{\sqrt{\pi}n!}\frac{d^n}{d\gamma^n}\left(e^{-i\gamma x}\left(i-\gamma\right)^n \right)
\]
evaluated at $\gamma=-i$. \\

We are now looking for a set of functions in $\mathbf L_\Delta$ whose images under $I$ constitute an orthogonal basis in $\mathcal G$. In view of \eqref{eq:def_h}, the function $(\gamma+i)\overline{h}s$, when $s$ belongs to the Schwartz space $\mathscr S$ of smooth rapidly decreasing functions, belongs to $\mathbf L_2(\mathbb R)$. Moreover
the linear span of the functions $(\gamma+i)\overline{h}s$ with $s\in\mathscr S$ is dense in $\mathbf L_2(d\gamma)$. Indeed, let $g\in\mathbf L_2(d\gamma)$ be such
that
\[
\int_{\mathbb R}s(\gamma)(\gamma+i)\overline{h}(\gamma)\overline{g}(\gamma)d\gamma=0,\quad \forall s\in\mathscr S.
\]
Then the function $(\gamma-i)h\overline{g}$ (which need not belong to $\mathbf L_2(d\gamma)$) defines the zero distribution on $\mathscr S$, and so is a.e. equal to $0$.\\
This proves that there exists a sequence $\left\{s_k\right\}:=\left\{s_k,\,k\in \mathbb N\right\}$ of Schwartz functions such that
\[
\lim_{k\rightarrow\infty}\|(\gamma+i)\overline{h}s_k-e_n\|_{\mathbf L_2(d\gamma)}=0.
\]
Therefore the sequence $\left\{s_k\right\}$ tends to $\frac{e_n}{(\gamma+i)\overline{h}}$ in $\mathbf L_2(d\Delta)$,
and so the sequence $\left\{\widecheck{s_k}\right\}$ is a Cauchy sequence in $\mathbf L_\Delta$. We denote its limit
by $\xi_n$.\\

If $\frac{e_n}{(\gamma+i)\overline{h}}$ is in $\mathbf L_2(d\gamma)$, then $\widehat{\xi_n}=\frac{e_n}{(\gamma+i)\overline{h}}$.

\begin{theorem} \label{th:orth_basis}
The set $\left\{ I(\xi_n),~n\in \mathbb Z\right\}$ forms an orthonormal basis for $\mathcal G$. Moreover,
\[
    \mathbb E\left[  I(\xi_n)|\mathscr F_{(-\infty,0]}\right]= \begin{cases}
                                                             I(\xi_n), & n<0 \\
                                                            0,  & n\geq 0
                                                            \end{cases},
\]
so that $\left\{ I(\xi_n),~n<0 \right\}$ spans the past, and $\left\{I(\xi_n),~n\geq 0 \right\}$ spans its orthogonal complement.
\end{theorem}

\begin{proof}
The fact that the $ I(\xi_n)$ are orthonormal is immediate by construction since
\[
\mathbb E\left[ I(\xi_n)  I(\xi_m)\right]=\int_{-\infty}^\infty \frac{e_n(\gamma)}{(\gamma-i)\overline{h}(\gamma)}\frac{\overline{e_m}(\gamma)}{(\gamma+i)h(\gamma)} (1+\gamma^2)|h(\gamma)|^2 d\gamma=\left(e_n,e_m\right)_{\mathbf L_2(d\gamma)}.
\]
To show that they span $\mathcal G$, let $f \in \mathbf L_\Delta$ and assume that $ I(f)$ is perpendicular to their span. Then for all $n \in \mathbb Z$,
\[
0=\mathbb E\left[ I(\xi_n)  I(f)\right]=\int_{-\infty}^\infty e_n (\gamma-i) \overline{\widehat{f}}{h} d\gamma=\left(e_n,(\gamma+i)\overline{h}\widehat{f}\right)_{\mathbf L_2(d\gamma)}.
\]
Since $\left\{e_n,\,\, n\in \mathbb Z\right\}$ is an orthonormal basis for $\mathbf L_2(d\gamma)$, it follows that $\widehat{f}~\overline{h}$ is zero in $\mathbf L_2(d\gamma)$. But it follows from condition \eqref{eq:szego1} that $h(\gamma)\neq 0$ almost everywhere, so we conclude that $f$, and thus $I(f)$, equals zero. Note that this also proves that $\left\{\xi_n,\,\, n\in \mathbb Z\right\}$ is an orthonormal basis of $\mathbf L_\Delta$.\\
Now for $t\leq 0$,
\begin{equation} \label{eq:proj}
\begin{split}
\mathbb E\left[I(\xi_n)X(t)\right] & =\int_{-\infty}^\infty \frac{e_n}{(\gamma+i)\overline{h}}(\gamma^2+1)\overline{z_t} \left|h(\gamma)\right|^2 d\gamma \\
& = \int_{-\infty}^\infty  (\gamma-i)e_n \overline{z_t}h d\gamma \\
& = \left(e_n,(\gamma+i)z_t\overline{h} \right)_{\mathbf L_2(d\gamma)}.
\end{split}
\end{equation}
From Lemma~\ref{lem:span} we know that the span of $\left\{ z_t(\gamma+i)\overline{h},\,\,t \leq 0  \right\}$ in $\mathbf L_2(\mathbb R)$ equals $H^{2-}$, so the last term in \eqref{eq:proj} vanishes for $n\geq0$. In order to calculate $\widetilde{I(\xi_n)}\triangleq \mathbb E\left[I(\xi_n)|\mathscr F_{(-\infty,0]}\right]$ for $n<0$ we can use the trigonometric isomorphism and instead look for the projection of $\frac{e_n(\gamma)}{(\gamma+i)\overline{h}}$ onto $\mathbf Z_{(-\infty,0]}$. Let $g=c_1 z_t+ \ldots +c_nz_t$, where $t_1,\ldots t_n, \leq 0$ and $c_1,...,c_n \in \mathbb C$.
\[
\begin{split}
\mathbb E\left[|I(\xi_n)-\widetilde{I(\xi_n)}|^2\right]=\inf_{g}\int_{-\infty}^\infty|
\frac{e_n(\gamma)}{(\gamma+i)\overline{h}(\gamma)}- g(\gamma)|^2 \Delta'(\gamma) d\gamma= \\
\inf_{g}\int_{-\infty}^\infty
|e_n(\gamma)- g(\gamma)(\gamma+i)\overline{h}(\gamma)|^2 d\gamma.
\end{split}
\]
Since the span of $\left\{ z_t(\gamma+i)\overline{h},\,\,t \leq 0  \right\}$ is $H^{2-}$ and $e_n \in H^{2-}$ for $n<0$, the last projection norm is trivial, so $\widetilde{I(\xi_n)}=I(\xi_n)$ in $\mathcal G$, and hence $I(\xi_n) \in \mathscr F_{(-\infty, 0]}$.
\end{proof}

\begin{Ex}
The chaos expansion of $X(t)$ is given by
\[
 X(t)=I(\mathbf 1_t)=\sum_{n=-\infty}^\infty \left(\xi_n,\mathbf 1_t \right)_\Delta  I\left(\xi_n\right)=\sum_{n=-\infty}^\infty \left(\xi_n,\mathbf 1_t \right)_\Delta H_{\epsilon(n)},
\]
where $\epsilon(n)=\left(...,0,1,0,...\right)$ with $1$ at the $n_{th}$ place. It follows that
\begin{equation} \label{eq:conditional_expansion}
\mathbb E\left[X(t)|\mathscr F_{(-\infty,0]} \right]=\sum_{n=-\infty}^{-1} \left(\xi_n,\mathbf 1_t \right)_\Delta I\left(\xi_n\right).
\end{equation}

The complementary projection is given by
\[
 \sum_{n=0}^{\infty} \left(\xi_n,\mathbf 1_t \right)_\Delta I\left(\xi_n\right),
\]
so that the variance of the prediction error is
\[
\sum_{n=0}^{\infty} \left(\xi_n,\mathbf 1_t \right)^2_\Delta.
\]
\end{Ex}

\begin{Ex}
The Wick exponent of the process $X(t)$ has the chaos expansion
 \[
:e^{ X(t)}: = \exp \left\{ I(\mathbf 1_t)-\| \mathbf 1_t \|_{\Delta}^2 \right\} = \sum_{\alpha\in \mathcal J} c_\alpha H_\alpha,
 \]
 with (see \cite[Exercise 2.8 (e)]{new_sde}
 \[
 c_\alpha = \prod_{n=-\infty}^\infty \frac{\left(1_t,\xi_n\right)_\Delta^{\alpha_n}}{\alpha_n!}.
 \]
 It follows that
 \[
 \mathbb E\left[:e^{X(t)}:| \mathscr F_{(-\infty,0]}  \right]=\sum_{\alpha\in \mathcal J_0} c_\alpha H_\alpha,
 \]
 where now for $\alpha \in \mathcal J_0$,
 \[
 c_\alpha = \prod_{n=-\infty}^{-1} \frac{\left(1_t,\xi_n\right)_\Delta^{\alpha_n}}{\alpha_n!}.
 \]

From \cite{MR99f:60082} we have
\[
\mathbb E\left[ :e^{ U}:| \mathscr F' \right]=:e^{ \mathbb E\left[U |\mathscr F' \right]}:=\exp\left\{\mathbb E\left[U | \mathscr F' \right]-\frac{1}{2}\mathbb E\left[X^2 | \mathscr F' \right]\right\}
\]
and we obtain the following identity:
\begin{equation}
\sum_{\alpha\in \mathcal J_0} c_\alpha H_\alpha= \exp \left\{\sum_{n=-\infty}^{-1} \left(\xi_n,\mathbf 1_t \right)_\Delta I\left(\xi_n\right)-\frac{1}{2}\sum_{n=0}^{\infty} \left(\xi_n,\mathbf 1_t \right)^2_\Delta\right\}.
\end{equation}
\end{Ex}
%We note that the Wick exponent of the Brownian motion, the so called \textit{geometric Brownian motion}, plays a special role in mathematical finance. In recent years, some attempts were made to replace the Brownian motion by Gaussian processes with correlated increments, and in particular by the fractional Brownian motion \cite{new_sde, vanDerHawk}.

\subsection{Basis Elements and Sample Path Relation}
In the classical prediction problem we are asked to find the conditional expectation with respect to $\mathscr F_{(-\infty,0]}$ expressed in term of the path $X_{(-\infty,0]}(\cdot) = \left\{X(t),\,t\in A \right\}$. By Theorem \ref{th:orth_basis} for $n<0$,
$I(\xi_n)$  is completely determined by $X_{(-\infty,0]}(\cdot)$. 
Due to the trigonometric isomorphism we have
\[
I(\xi_n)=I\left(\mathbf 1_{[-\infty,0]}\xi_n\right)+I\left(\mathbf 1_{[0,\infty]}\xi_n\right),
\]
and
\begin{equation} \label{eq:xi_is_zero}
 \|1_{[0,\infty]}\xi_n\|_\Delta^2=\int_{-\infty}^\infty  \left|\widehat{1_{[0,\infty]}\xi_n}(\gamma) \right|^2 d\Delta(\gamma).
\end{equation}
Note that $\widehat{1_{[0,\infty]}\xi_n}$ is the projection of
\begin{equation} \label{eq:proj_2_H2}
\widehat{\xi_n}(\gamma)=\frac{e_n(\gamma)}{(\gamma+i)\overline{h}(\gamma)}= \frac{1}{\sqrt{\pi}} \left(\frac{1+i\gamma}{1-i\gamma}\right)^n  \frac{ 1}{(1+\gamma^2)\overline{h}(\gamma)}
\end{equation}
into $H^{2+}$ (see \cite[Ch. 2.4]{Dmk}). If $\overline{h}$ does not vanish too fast as $\gamma \rightarrow \infty$, (in general, condition \eqref{eq:szego1} does not guarantee that), then $\widehat{\xi_n}$ belongs to $H^{2-}$, in which case $\widehat{1_{[0,\infty]}\xi_n}=0$. That is, we have the following proposition:
\begin{prop} \label{prop:path}  If $\widehat{\xi_n} \in \mathbf L_2(d\gamma)$, then 
\[
\mathbb E \left[ I(\xi_n) | \mathscr F_{(-\infty,0]}\right] =  I(\xi_n \mathbf 1_{(-\infty,0]}),
\]
 for all $n=-1,-2,\ldots$.
\end{prop}
If the condition in Proposition~\ref{prop:path} is met, then the $\mathbf L_2\left(\Omega,\mathscr F_{\mathbb R},\mathbb P\right)$ stochastic integral with respect to the process $X(t)$ can also be evaluated from its sample path which is defined in the standard way as in explained in \eqref{eq:Wiener_integral}. See also \cite[Sec. 2]{Gripenberg96} for a pathwise definition of the stochastic integral with respect to the fractional Brownian motion. In such case we get
\begin{equation} \label{eq:fBm_int}
I(\xi_n) = I(\mathbf 1_{(-\infty,0]}\xi_n)=\int_{-\infty}^0 \xi_n(t) dX(t).
\end{equation}
This allows us to express the $I(\xi_n)$ in \eqref{eq:conditional_expansion} in terms of the sample path, which leads to a prediction formula for a general Gaussian stationary increment process.

\section{Bounded time interval} \label{sec:bti}

In the case where $A=[-T,T]$ we are looking for an orthogonal basis for $\mathcal G_{[-T,T]}$, or equivalently,
for its image under the trigonometric isomorphism $\mathbf Z_{[-T,T]}$. For the conditions to optimal prediction in this case we refer to \cite[Section 6.4]{Dmk}. The space $\mathbf Z_{[-T,T]}$ is a reproducing kernel Hilbert space of entire functions isometrically included in $\mathbf L_2(d\Delta)$, and invariant under
the backward shift operators
\[
R_af(\gamma)=\frac{f(\gamma)-f(a)}{\gamma-a},\quad a\in\mathbb C.
\]
Therefore, by a theorem of de Branges, see \cite[Theorem 3]{dbhsaf1}, and (for instance) by an application of
\cite[Theorem 3.1]{ad1}, one sees that the reproducing kernel of $\mathbf Z_{[-T,T]}$ is of the form
\begin{equation}
\label{dbrk}
K(\gamma,\lambda)=\frac{B(\gamma)\overline{A}(\lambda)-A(\gamma)\overline{B}(\lambda)}{\gamma-\overline{\lambda}},
\end{equation}
where $A(\gamma)$ and $B(\gamma)$ are entire function of the variable $\lambda$ of finite exponential type.
A characterization of certain orthogonal sets in such spaces is given in \cite[Theorem 22]{MR0229011}. We recall the
result for completeness. Set $E(\gamma)=A(\gamma)-iB(\gamma)$. Then there exists a continuous function
$\varphi(x)$ ($x\in\mathbb R$) such that $E(x,T)e^{i\varphi(x)}\in\mathbb R$ for all $x\in\mathbb R$. Let
$\alpha\in\mathbb R$ and let $x_1,x_2,\ldots \in\mathbb R$ be such that $\varphi(x_n)\equiv \alpha$ (mod $\pi$).
The functions
\[
K(\gamma, x_n)
\]
form an orthogonal set of $\mathbf Z_{[-T,T]}$, and it is complete if and only if the function $e^{-i\alpha}E(\gamma)-e^{i\alpha}
\overline{E}(\overline{\gamma})$ does not belong to $\mathbf Z_{[-T,T]}$.\\

One can compute explicitly the functions $A(\gamma)$ and $B(\gamma)$ in some special cases. For instance
Dym and Gohberg considered in \cite{dg-80} the case where the spectral density $\Delta^\prime$ is the form
\[
\Delta^\prime(\gamma)=1-\widehat{k}(\gamma),
\]
where $k\in\mathbf L^1(\mathbb R)$ and such that $1-\widehat{k}>0$
(in fact, they consider the matrix-valued non Hermitian case). The case where
\[
\Delta^\prime(\gamma)= c_H |\gamma|^{1-2H},
\]
which corresponds to the fractional Brownian motion, was considered by Dzhaparidze and H. van Zanten in \cite{DVZ} and will be revisited again in Section \ref{sec:fbm_bounded}.
More generally, one needs to use Kreins's theory of strings, as explained in \cite{Dmk} and \cite[Section 2.8]{DVZ}, to compute the reproducing kernel of $\mathbf Z_{[-T,T]}$. \\

%In order to use the same notions as in Sections \ref{sec:proof_main} and \ref{sec:entire_past}, we set $J_0={-1,-2,...}$ and let
%\[
%E_n(\gamma)=K(\gamma,x_{-n}),\quad n\in J_0.
%\]
%So $\left\{ E_n, \, n\in J_0 \right\}$ is an orthogonal basis for $\mathbf Z_{[-T,T]}$ and we may complete it to obtain an orthogonal basis  $\left\{F_n,\,\, n\in J\right\}$ for $\mathbf Z$, with $J_0 \subset J \subset \mathbb Z$ and $F_n=E_n$ for $n\in J_0$. We also take $\xi_n$ to be the inverse image of $E_n$ under the trigonometric isomorphism. Without loss of generality we may take $J=\mathbb Z$,  By replacing $\mathbf Z_{(-\infty,0]}$ with $\mathbf Z_{[-T,T]}$ and $\mathscr F_{(-\infty,\infty]}$ with $\mathscr F_{[-T,T]}$, we are in a similar situation as in Theorem \ref{th:orth_basis}.

\section{Chaos and Prediction with respect to the Fractional Brownian Motion}
\label{sec:application}
\setcounter{equation}{0}

In this section we now specialize to the case where the spectral measure of $X$ is given by
\begin{equation} \label{eq:fBm_spectral}
d\Delta(\gamma)= C_H |\gamma|^{1-2H} d\gamma,\quad 0<H<1,
\end{equation}
where $C_H = \frac{\Gamma(1+2H) \sin \left(\pi H\right)} {2 \pi}$ and $\Gamma(x)$ is Euler's Gamma function. This measure satisfies conditions \eqref{eq:measure_growth_condition} and \eqref{eq:szego1}. The corresponding stationary increment Gaussian process $X$ is called the fractional Brownian motion with Hurst parameter $H$ and is denoted $B_H$. Its covariance function is given by
\[
\mathbb E\left[ B_H(t) B_H(s) \right]= \frac{1}{2}\left( |t|^{2H}+|s|^{2H}-|t-s|^{2H} \right),
\]
and it can be shown to have an almost surely continuous sample paths \cite{Molchan}. This process has been found useful in a host of applications, and was extensively studied in the past few decades; see for example \cite{bosw, Ustunel, MR1790083, duncan2009control}. \\

In what follows, we will evaluate the coefficients in the chaos expansion for $B_H(\cdot)$ based on an orthonormal basis with the properties of Theorem \ref{th:main2}. We start with the case of prediction with respect to the entire past. 

\subsection{Prediction with respect to the Entire Past ($A=\left(-\infty,0\right]$)}
For this process $B_H(\cdot)$  with $A=\left(-\infty,0\right]$ we derive the following:

\begin{theorem}
The outer function in the Wiener-Hopf spectral decomposition of $\frac{\Delta'(\gamma)}{1+\gamma^2}$, where $\Delta'(\gamma)=C_H |\gamma|^{1-2H}$, that admits the reality condition $(-\gamma-i)h(-\gamma)=(\gamma+i)\overline{h}(\gamma)$ is given by
\[
h(\gamma) =\sqrt{C_H}\exp\left\{i \pi \frac{2H-1}{4}sign(\gamma)\right\} \frac{i-\gamma}{1+\gamma^2} |\gamma|^{1/2-H}.
\]
\end{theorem}

\begin{proof}
Using the formula \cite[(11), p. 193]{duren}, to find the outer factor $\tilde{h}$ in the decomposition \eqref{eq:def_h},
%I used exercise 2.7.4 in DynMcKean
 we see that the outer factor in the factorization of \eqref{eq:fBm_spectral} is given by
%The integral in \cite{Karhunen (1950)} \cite[p. 161]{Kailath3Decades}
\begin{equation} \label{eq:fBm_outer_integral}
\begin{split}
\phi(z) & = \exp \left\{\frac{1}{2\pi  i} \int_{-\infty}^\infty  \frac{\gamma z +1}{\gamma -z} \frac{\log(C_H |\gamma |^{1-2 H} / (\gamma^2+1) )} {\gamma^2+1} d\gamma   \right\} \\
& = \exp\left\{ -i\left(\pi/2-\tan^{-1}(z)-z\sqrt{-1/z^2} \tanh^{-1} (1+\frac{2}{z^2}) \right) \right\} \\
\times & \exp \left\{ -\frac{i}{4\pi} \left( \log(-z)-\log(z) \right)\left( (2H-1)\left(\log(z)+\log(-z) \right)-2\log(C_H)  \right) \right\}
\end{split}
\end{equation}
Taking the restriction of $\phi(z)$ to the real line we obtain
\[
\frac{i\gamma+1}{1+\gamma^2}  ~ \sqrt{C_H} \exp\left\{sign(\gamma) \frac{2H-1}{4} \pi i\right\} |\gamma|^{1/2-H}.
\]
To finish the proof we multiply the last term by $i$ in order to impose the reality condition
\[
(-\gamma-i)h(-\gamma)=(\gamma+i)\overline{h}(\gamma) \Longleftrightarrow \overline{h}(-\gamma)=-\overline{h}(\gamma).
\]
\end{proof}

The deterministic coefficients in the chaos expansion \eqref{eq:conditional_expansion} for the conditional expectation $\mathbb E\left[ B_H(t)| \mathscr F_{(-\infty,0]} \right]$ are given by
\[
\begin{split}
r^H_j(t) \triangleq \left(\mathbf 1_t,\xi_j\right)_\Delta & = \int_{-\infty}^\infty z_t(\gamma) \frac{\overline{e_j}(\gamma)}{h(\gamma)(\gamma-i)} (1+\gamma^2)|h(\gamma)|^2d\gamma \\
& =\frac{1}{\sqrt{\pi C_H}} \int_{-\infty}^\infty \frac{e^{i\gamma t}-1}{\gamma}   \left(\frac{1-i\gamma}{1+i\gamma}\right)^j \\
& \times \exp\left\{i \pi \frac{2H-1}{4}sign(\gamma)\right\}  |\gamma|^{H-0.5}  d\gamma.
\end{split}
\]
See Figure \ref{fig:Coeff} for a graphical illustration of these coefficients for a few cases of the Hurst parameter $H$.

\begin{figure}
        \begin{subfigure}[b]{0.3\textwidth}
                \centering
                \includegraphics[scale=0.3]{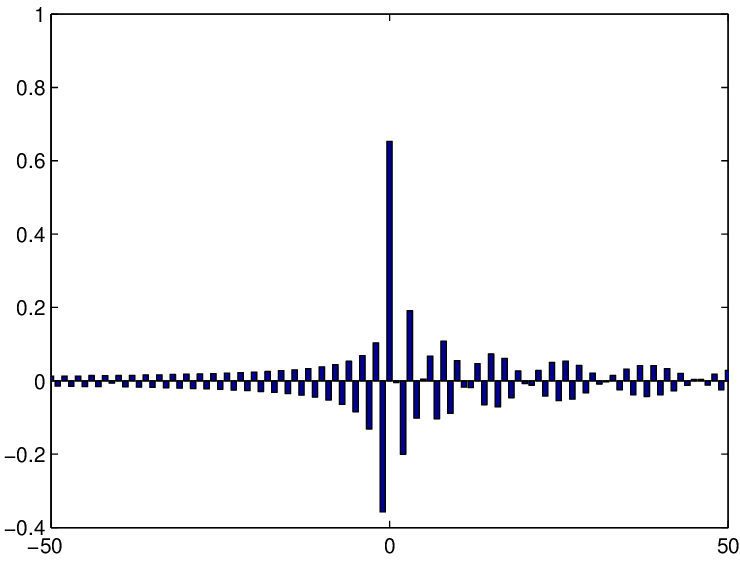}
                \caption{H=0.2}
        \end{subfigure}%
        \begin{subfigure}[b]{0.3\textwidth}
                \centering
                \includegraphics[scale=0.3]{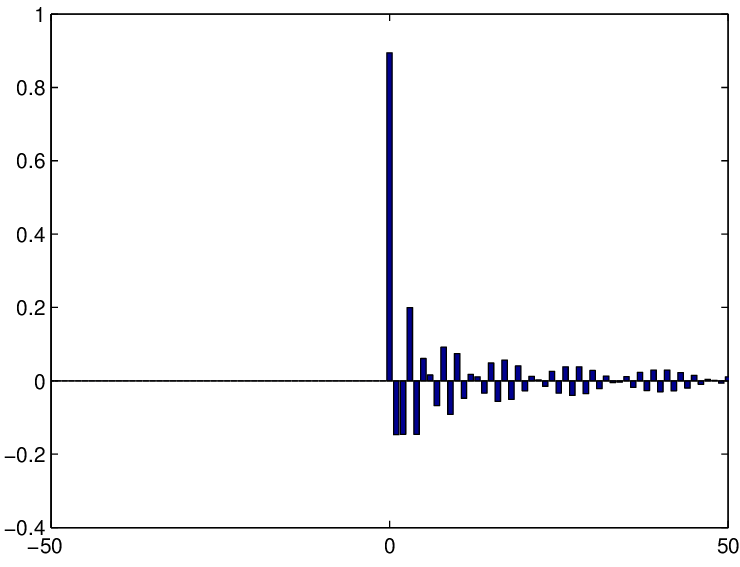}
                \caption{H=0.5}
        \end{subfigure}
        \begin{subfigure}[b]{0.3\textwidth}
                \centering
                \includegraphics[scale=0.3]{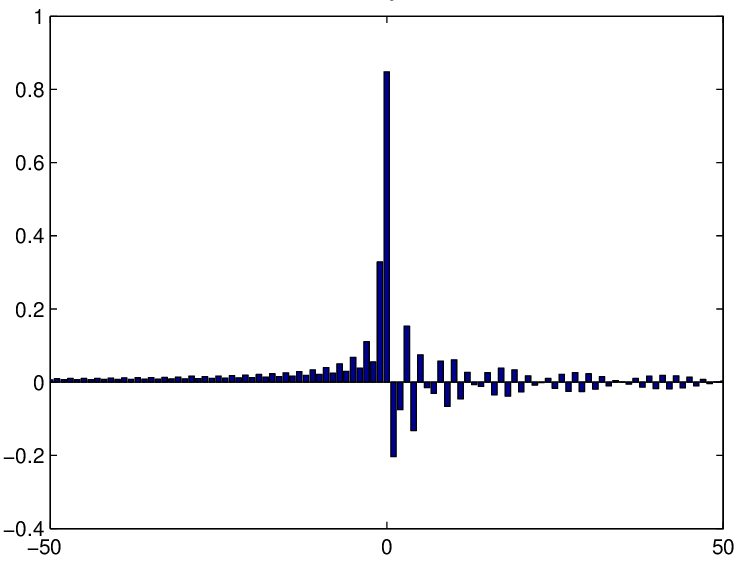}
                \caption{H=0.7}
        \end{subfigure}
        \caption{MAgnitudes of the coefficient $r^H_j(1)=\left( \mathbf 1_{[0,1]}, \xi_j \right)_\Delta$, $-100\leq j \leq 100$, in the sum \eqref{eq:PaleyWiener_fBm} for different values of the Hurst parameter $H$ at time $t=1$. $H=0.5$ corresponds to the Brownian motion.}  \label{fig:Coeff}
\end{figure}
This gives us
\begin{equation} \label{eq:PaleyWiener_fBm}
 B_H(t)=\sum_{j\in \mathbb Z} r^H_j(t) I(\xi_j),
 \end{equation}
which is a representation of the process $B_H(t)$ as a sum of mutually orthogonal Gaussian random variables independent of time, weighted by the coefficients $\left\{ r^H_j(t),\, j\in \mathbb Z\right\}$. %This can be seen as a generalization to the expansion used by Paley and Wiener as the definition of the Brownian motion \cite{MR1451142}. 

It can be shown that the sum \eqref{eq:PaleyWiener_fBm} converges in $\mathbf L_2 (\Omega)$ uniformly in $t\in \mathbb R$. We now split the sum in \eqref{eq:PaleyWiener_fBm} 
into two sums, of elements of $j<0$ and $j\geq0$ respectively, which by \eqref{eq:conditional_expansion} corresponds to the sum of two (in general not orthogonal) processes: 
\begin{equation} \label{eq:fbm_past}
\sum_{j\leq -1}r_j^H(t) I(\xi_j) =   \mathbb E \left[B_H(t) |\mathscr F_{(-\infty,0]}\right],
\end{equation}
and 
\begin{equation} \label{eq:fbm_future}
\sum_{j\geq 0} r_j^H(t) I(\xi_j) = B_H(t) -  \mathbb E \left[B_H(t) | \mathscr F_{(-\infty,0]} \right].
\end{equation}
For $t<0$, $r_j^H=0$ for any $j\geq0$, and $B_H(t)$ coincides with \eqref{eq:fbm_past}. For $t>0$, the sum in \eqref{eq:fbm_past} represents what the past `thinks' the future looks like given a specific realization, i.e. the projection of the future on the past. The other part of the sum \eqref{eq:fbm_future} is the complementary projection which can only be determined by the future. Figure \ref{fig:sample_path} illustrates a realization of $B_H(\cdot)$ rendered according to the two sums in \eqref{eq:fbm_past} and \eqref{eq:fbm_future}. From \eqref{eq:fbm_future} we also obtain an expression for the prediction error in estimating $B_H(t)$ for $t>0$:
\begin{equation} \label{eq:error_fbm}
\mathbb E\left(B_H(t) -  \mathbb E \left[B_H(t)|\mathscr F_{(-\infty,0]} \right] \right)^2 = \sum_{j\geq 0} \left(r_j^H(t) \right)^2.
\end{equation}
A closed form expression for $\mathbb E\left(B_H(t) -  \mathbb E \left[B_H(t) |\mathscr F_{(-\infty,0]} \right] \right)^2$
was derived in \cite{Gripenberg96}, which leads to the identity
\begin{equation} \label{eq:fbm_error_exact}
\sum_{j\geq 0} \left(r_j^H(t) \right)^2 =  \frac{\sin\left(\pi(H-\frac{1}{2})\right) \Gamma\left(\frac{3}{2}-H\right)^2} {\pi\left(H-\frac{1}{2}\right) \Gamma \left(2-2H\right)}t^{2H}.
\end{equation}

\begin{figure} 
\begin{center}
  \includegraphics[scale=0.75, trim = 0.5cm 0cm 2cm 0cm, clip=true]{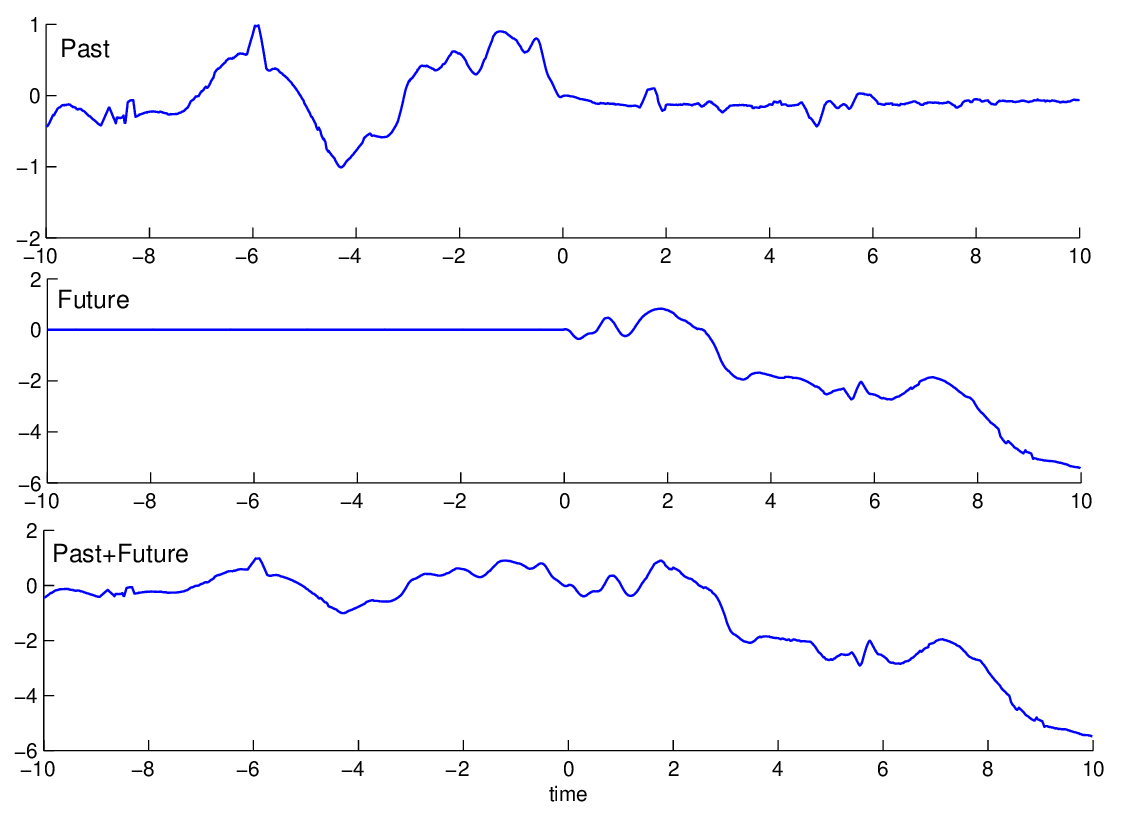}
  \caption{\label{fig:sample_path} An illustration of a single sample path of the fractional Brownian motion $B_H$ with $H=0.7$, rendered according to the two components of the sum \eqref{eq:PaleyWiener_fBm}: The `past' and `future' correspond to the negative and non-negative indexes, respectively. }
  \end{center}
\end{figure}

%\begin{center}
%\begin{figure} 
%\begin{tikzpicture}
%  \node at (0,0) {\includegraphics[scale=0.5, trim = 0.7cm 0.7cm 0cm 0cm, clip=true]{coeff_err.eps} };
%  \node at (0,-3.5) {number of coefficients};
%  \node[rotate=90] at (-6.1,0) {error};
%  \end{tikzpicture}
%  
%\caption{\label{fig:sample_path_square} Error in estimating $B_H(1)$ with $H=0.7$ as a function of the number of coefficients in the sum \eqref{eq:fbm_past}. The dashed curve is the exact error obtained by \eqref{eq:fbm_error_exact}. }
%\end{figure}
%\end{center}

The significance of Theorem \ref{th:main2} is in simplifying expressions for the conditional expectation of non-linear functions of $B_H(t)$. For example, the chaos expansion of $B_H(t)^2$ is found by the Wick product identity \cite[Eq. 2.4.10]{new_sde}:
\begin{equation}
\begin{split}
B_H^2(t)- \left| t \right|^{2H}  & =  : B_H(t), B_H(t):  \\
& = \sum_{j} \left(r^H_j(t)\right)^2 h_2(E_j) + \sum_{i\neq j} r^H_i(t) r^H_j(t) h_1(E_i)h_1(E_j),
\end{split}
\end{equation}
where $i$ and $j$ go over all integers ,and we used the fact that \[
\sum_{j\in \mathbb Z} \left(r^H_j(t) \right)^2 = \mathrm{var}\left( B_H(t)\right) = \left| t \right|^{2H}.
\]
From Theorem~\ref{th:main2} we conclude
\begin{equation}
\begin{split} \label{eq:fbm_square}
\mathbb E \left[B_H^2(t)| \mathscr F_{(-\infty,0]} \right] & = \\ 
 \left| t \right|^{2H} +  \sum_{j\leq -1} & \left(r^H_j(t)\right)^2 h_2(E_j) + \sum_{i\neq j \leq -1} r^H_i(t) r^H_j(t) h_1(E_i)h_1(I(E_j)).
\end{split}
\end{equation}
%A sample path of $\mathbb E\left[B_H^2(t) | \mathscr F_{(-\infty,0]} \right]$ obtained using \eqref{eq:fbm_square} is illustrated in Figure~\ref{fig:sample_path_square}.  
%\label{eq:fbm_square}
%\begin{figure} 
%\begin{center}
%  \includegraphics[scale=0.5, trim = 0cm 0cm 0cm 0cm, clip=true]{fBm_path_square.eps}
%  \caption{\label{fig:sample_path_square} One sample path of $B_H^2(t)$ (red) and $\mathbb E\left[B_H^2(t) | \mathscr F_{(-\infty,0]} \right]$ (dashed blue) with $H=0.7$, rendered according to \eqref{eq:fbm_square}. }
%  \end{center}
%\end{figure}

\subsection{Prediction with respect to $A={[-T,T]}$ \label{sec:fbm_bounded}}
An orthonormal basis $\left\{\xi_n\,\,n\in \mathbb Z\right\}$ for the space $\mathbf Z_{[-T,T]}$ in the case of the fractional Brownian motion was obtained in \cite{MR2178502}. This basis is defined in terms of the zeros of $J_{1-H}$, which is the Bessel function of the first order with parameter $1-H$. Specifically
\begin{equation}
\xi_n (\gamma)=\frac{S_T(2\gamma_n/T,\gamma)}{\| S_T (2\gamma_n/T,\cdot)\|_{\Delta}},
\end{equation}
where $...<\gamma_{-1}<\gamma_0=0<\gamma_1<...$ are the zeros of $J_{1-H}$. In addition
\[
\begin{split}
S_T(2\eta,2\gamma)=S_T(0,0)(2-2H)\Gamma^2(1-H)\left(\frac{T^2\gamma\eta}{4} \right)e^{iT(\gamma-\eta)} \\
\times \frac{J_{-H}(T\eta)J_{1-H}(T\gamma)-J_{1-H}(T\eta)J_{-H}(T\gamma)}{T(\gamma-\eta)},
\end{split}
\]
for $\eta \neq \gamma$, and
\[
\begin{split}
S_T(2\gamma,2\gamma)=S_T(0,0)(2-2H)\Gamma^2(1-H)\left(\frac{T\gamma}{2} \right) \\
\times\left( J^2_{1-H}(T\gamma)+\frac{2H-1}{T\gamma}J_{-H}(T\gamma)J_{1-H}(T\gamma)+J^2_{-H}(T\gamma)\right),
\end{split}
\]
for $\gamma \in \mathbb R$.\\

\section{Concluding Remarks}
In this work we combined the Wiener chaos decomposition with the problem of linear prediction for Gaussian stationary-increment processes to derive the chaos decomposition of a general non-linear function of the process conditioned on past realizations of the process. This decomposition is obtained by considering a special basis for the Gaussian Hilbert space generated by the process, in which each basis element is either completely measurable with respect to the observations or independent of it. This special basis has the property that each basis element is either measurable or independent with the $\sigma$-field generated by the observations. The result is a chaos approach to prediction, which can be employed to easily derive an expression for the chaos expansion of any non-linear function of the processes with respect to past observations. \\

\bibliographystyle{plain}
\bibliography{chaos}

\end{document}